\documentclass[letterpaper, 10 pt, conference]{ieeeconf}  

\IEEEoverridecommandlockouts                              
\usepackage{amsthm}

\usepackage{subcaption}
\usepackage{graphicx}
\usepackage{hyperref}       
\usepackage{url}            
\usepackage{booktabs}       
\usepackage{amsfonts}       
\usepackage{amsmath,bm}
\usepackage{algorithm,algorithmic}
\usepackage{commath}
\usepackage{mathfont}

\DeclareMathOperator*{\argmax}{arg\,max}
\DeclareMathOperator*{\argmin}{arg\,min}

\newtheorem{theorem}{Theorem}
\newtheorem{remark}{Remark}
\newtheorem{lemma}{Lemma}

\usepackage{xcolor}

\newcommand{\xt}{x_{t}}
\newcommand{\xti}{x_{t}^{(i)}}
\newcommand{\lip}{\mathrm{Lip}}

\newcommand{\utheta}{u_{\theta}}
\newcommand{\phieta}{\phi_{\eta}}
\newcommand{\phietatheta}{\phi_{\eta(\theta)}}

\title{\LARGE \bf
High-dimensional Optimal Density Control with Wasserstein Metric Matching
}

\author{Shaojun Ma \and Mengxue Hou \and Xiaojing Ye \and Haomin Zhou
\thanks{This work was supported in part by NSF under grants DMS-1925263 and DMS-2152960,
and ONR N00014-21-1-2891.}
\thanks{S.\ Ma and H.\ Zhou are with the School of Mathematics, Georgia Tech, Atlanta, GA 30332, USA
        {\tt\small \{shaojunma,hmzhou\}@gatech.edu}}%
\thanks{M.\ Hou is with the College of Engineering, Purdue University, West Lafayette, IN 47906, USA
        {\tt\small hou178@purdue.edu}}%
\thanks{X.\ Ye is with the Department of Mathematics and Statistics, Georgia State University,
        Atlanta, GA 30303, USA
        {\tt\small xye@gsu.edu}}%
}

\begin{document}

\maketitle
\thispagestyle{empty}
\pagestyle{empty}

\begin{abstract}

We present a novel computational framework for density control in high-dimensional state spaces. The considered dynamical system consists of a large number of indistinguishable agents whose behaviors can be collectively modeled as a time-evolving probability distribution. The goal is to steer the agents from an initial distribution to reach (or approximate) a given target distribution within a fixed time horizon at minimum cost. To tackle this problem, we propose to model the drift as a nonlinear reduced-order model, such as a deep network, and enforce the matching to the target distribution at terminal time either strictly or approximately using the Wasserstein metric. The resulting saddle-point problem can be solved by an effective numerical algorithm that leverages the excellent representation power of deep networks and fast automatic differentiation for this challenging high-dimensional control problem. A variety of numerical experiments were conducted to demonstrate the performance of our method.

\end{abstract}

\section{INTRODUCTION}

Recent years have witnessed an emerging research interest in optimal control problems with large collective of agents, such as drones, robots, and vehicles \cite{brambilla2013swarm,dorigo2021swarm}. Such problems appear in many important real-world applications, including surveillance over a large region, time-sensitive search-and-rescue, infrastructure inspection and maintenance, among many others. Meanwhile, the cost to manufacture robots, including the builtin electrical components and batteries, has reduced significantly during the past decades. The individual robots used in these swarm control applications are often of small size and thus have limited computation and communication power. However, using such swarms of robots is particularly necessary and cost-effective to achieve goals in the aforementioned applications. Optimal control of swarm robots is to design efficient algorithms that can steer them to achieve specific tasks achieve specified high-level tasks in a way that accommodates the high dimensionality of the system \cite{sinigaglia2022robust}.

\subsection{Related Work}

Optimal control of large collective of agents has been commonly considered in the context of density control based on mean-field models
%
\cite{elamvazhuthi2019mean,zheng2022backstepping,lasry2007mean}. 
The idea of mean-field models is to characterize the collective behaviors of the agents in the swarm using a time-evolving probability density function defined on the specified state space. Mean-field models allow for both deterministic and even stochastic behaviors if the randomness of individual agents can be appropriately described by stochastic processes. The dynamics of the density function is governed by the associated continuity equation in the deterministic case and Fokker-Planck equation in the stochastic case, which are evolution partial differential equations (PDEs). In these models, the probability density characterizing the swarm behavior does not depend on the number of agents, but only on the state space of individual agents. Nevertheless, direct computation of these PDEs can still be computationally prohibitive due to the high dimension of the state space.
%
%
%
Thus, most xisting works on density controls approximate the density functions based on spatial discretization \cite{accikmecse2012markov} or using kernel density estimation \cite{bandyopadhyay2017probabilistic,eren2017velocity}. 
%

While the agents and the population dynamics are steered by the control, an optimal density control problem aims at finding the controls that achieve the specified task with minimum control cost. The work \cite{eren2017velocity} designed velocity field to make swarms density converges to the desired/commanded density distribution. In \cite{zhao2011density}, the authors presented a density control by improving Smoothed Particle Hydrodynamic (SPH) \cite{pimenta2008control} method with collision free condition, where one defines the density of a robot as the weighted sum of distances to its nearby robots within a certain range. A data-driven approach was proposed in \cite{choi2021convex} to learn the control by estimating the Koopman operators. In \cite{yi2020nonlinear}, the authors provided a control strategy to steer the state from an initial distribution to a terminal one with specified mean and covariance, which is also called nonlinear covariance control. For other more general works we refer to \cite{otcontrol1, otcontrol2, haasler2020optimal}. 

A specific type of density control problem, known as optimal transport, has received significant attentions in recent years. Several works have combined Wasserstein metric and optimal control to develop relevant control algorithms. For instance, \cite{caluya2021wasserstein} develops Wasserstein proximal algorithms to solve density control problem, where a nonlinear drift term is considered. In \cite{bonnet2021necessary}, the authors discussed the optimality conditions for optimal control problems in Wasserstein spaces. In \cite{yang2020wasserstein}, a Wasserstein based robust control was developed to resolve the issue that uncertain variables is unavailable. In \cite{hakobyan2021wasserstein}, the authors proposed a novel model-predictive control (MPC) method for limiting the safety risk when the distribution of the obstacles is set to be within an ambiguity set which is measured by the Wasserstein metric. In \cite{hoshino2020finite}, terminal cost as Wasserstein distance and finding the parameters of normalizing flow \cite{dinh2014nice} by minimizing the terminal cost have been considered. 
In \cite{elamvazhuthidynamical}, the authors presented a primal-dual formulation of optimal control problem with existence proof and efficient numerical method, similar method can also be found in \cite{gao2020energy} where the problem is solved by a Mean Field Game (MFG) approach. 
In \cite{frederick2022collective}, the authors calculated the individual robot trajectories by alternating two gradient flows that involve an attractive potential, a repelling function, and a process of intermittent diffusion, in which way a large group of robots are employed to accomplish the task of shape formation. 

Optimal control over discrete state spaces can be cast as Markov chain models as shown in \cite{accikmecse2012markov}. For optimal control with continuous state spaces, the density dynamics are determined by the corresponding continuity equation or Fokker-Planck equation. For instance, \cite{roy2018fokker,sinigaglia2022density} provided both theoretical analysis and numerical methods of the optimal control to achieve target equilibrium density. In \cite{chen2021density}, the dynamics of the control was discretized in time and a multi-marginal optimal transport scheme was used to solve the problem. In \cite{inoue2021optimal}, the authors developed a method to solve the optimal transport based control based on Voronoi Tessellation coverage. A distributed algorithm for density estimation to reduce the computational cost is proposed in \cite{zheng2021distributed}. 
Multi-agent optimal control has been widely applied in the domain of robotics in recent years. For a general introduction we refer to \cite{feng2020overview}. In particular, \cite{zhu2020multi} presents the work of using reinforcement learning to control multi-robot system. Futhermore, \cite{luis2020online, wang2017safety, onken2021neural} carefully designed an objective function to minimize the control energy, $L_2$ based difference between the initial and target positions, meanwhile keeping collision-free conditions. Parameterizing controls using neural networks was also considered in  \cite{onken2021neural}. 

\subsection{Main contributions of this work}

In this work, we propose a novel method to compute the optimal control $u$ which steers an initial density as close as possible to a target density at a prescribed terminal time, while minimizing the control cost. The major advantage of the proposed method is that it can be readily applied to the cases with high-dimensional state spaces, where computation of the associated continuity equation or Fokker-Planck equations is prohibitive using existing approaches. In our approach, this issue is effectively tackled by using a large number of synthesized agents whose trajectories can be easily evaluated using decoupled ODEs or SDEs. Therefore, our approach does not require any spatial discretization using finite difference and element methods or any basis representations, and the expected cost function values can be accurately approximated using empirical expectations, which allows our method to scale to much higher dimensions that are traditionally considered computationally infeasible. Moreover, the proposed method can readily adopt parallel computation because of the decoupling, and hence the computation can be done very fast by leveraging the power of standard deep learning computation framework. We also provide convergence analysis, including the rate of convergence, of the numerical algorithm for solving the resulting saddle-point problem. We validate our approach in numerical simulations that include complex scenarios in various tests with dimension up to 100.

\subsection{Paper organization}
The remainder of the paper is organized as follows. In Section \ref{sec:method}, we present the proposed framework to solve the optimal density control with Wasserstein metric matching to any target distribution at a prescribed terminal time. The framework renders a saddle-point problem of the control and the dual variable of the Wasserstein distance. We also provide convergence analysis of the proposed method including the convergence rate. In Section \ref{sec:result}, we conduct several numerical experiments on synthetic and real data to demonstrate the performance of the proposed algorithm. Section \ref{sec:conclusion} concludes this work.

\subsection{Notations}
Throughout this paper, we use $|\cdot|$ to denote the absolute value of a scalar and Euclidean norm of a vector. 
Furthermore, $\|\cdot\|$ denotes the matrix norm induced by vector Euclidean norm. We also use $\|\cdot\|_{\cdot}$ to denote function norms where the space is specified in the subscript $\cdot$. The spatial domain is denoted by $\Omega$, which is assumed to be a closed and quasiconvex subset of $\Rbb^{d}$. Integrals over $\Omega$ or $\Omega \times \Omega$ are written without subscript for notation simplicity unless otherwise specified. We use $\mu$ to denote standard Lebesgue measure on $\Rbb^{d}$.
We use $\nabla$ to denote the gradient with respect to the state variable $x$. Gradient and partial derivatives with respect to any other variables will be indicated by their subscripts.
We denote $\Nbb$ the set of positive integers, and $[n]:=\set{1,\dots,n}$ for notation simplicity.

\section{PROPOSED METHOD}
\label{sec:method}
In this section, we develop a general computational framework for solving density control over high-dimensional state spaces. 
To maximize applicability of our framework, we use the following prototype mean-field model of density control with any user-defined (running) cost functional $c$:
\begin{equation}
    \label{eq:dc-obj}
    \min_{u}\ \Ebb\sbr[2]{\int_0^{T} c(\xt,u(\xt))\, dt} + \gamma D(\rho_{T}, \nu)
\end{equation}
where $\xt \in\Omega$ is the dynamic following
\begin{equation}
    \label{eq:dc-sde}
    \begin{cases}
        d\xt = u(\xt)dt + \sigma(\xt) dW_{t}, \\
        x_{0} \sim \rho_{0},
    \end{cases}
\end{equation}
which represents any indistinguishable agent in the swarm to be controlled.
In \eqref{eq:dc-obj}, $\rho_0$ is a given initial population density, $\nu$ is a given target population density, $D$ is a metric that measures the distance between two probability densities, $\gamma > 0$ is a user-chosen weight parameter, and 
$\rho(t,\cdot)$ is the density of $\xt$ for any $t$ with short hands $\rho(0,\cdot) = \rho_0(\cdot)$ and $\rho(T,\cdot) = \rho_{T}(\cdot)$.
The time evolution of $\rho(t,\cdot)$ is governed by the Fokker-Planck equation
\begin{equation}
    \label{eq:fpe}
    \partial_{t} \rho =  - \nabla \cdot (\rho u) + \frac{1}{2} \langle \nabla^2, \sigma\sigma^{\top} \rho \rangle,
\end{equation}
where $\langle \nabla^2, A(x)\rangle := \sum_{i,j} \partial_{x_i x_j}^2 a_{ij}(x)$ for any matrix-valued function $A(x) = [a_{ij}(x)] \in \Rbb^{d\times d}$.
The goal of \eqref{eq:dc-obj} is to find the optimal control $u: \Omega \to \Rbb^{n}$
that steers the population characterized by $\xt$ with initial distribution $\rho_0$ to approximate the target distribution $\nu$ at terminal time $T$ with minimal overall cost.
The parameter $\gamma$ weighs the matching of terminal density to $\nu$ against the overall cost.
If an exact matching to $\nu$ at $T$ is desired, then $\gamma$ can be cast as the Lagrangian multiplier and solved jointly with $u$. For simplicity, we will only consider the case with soft-penalty on the matching as given in \eqref{eq:dc-obj} in the present work.
In what follows, we will discuss the choice of cost functional $c$ and density distance $D$ to instantiate \eqref{eq:dc-obj}.

\paragraph{Cost functional}
The density control problem \eqref{eq:dc-obj} allows for a general class of cost functional $c: \Omega \times \Rbb^n \to \Rbb$. Typical optimal control problems with minimal energy use $c(x,u) = \frac{1}{2}|u(x)|^2$. If there is a component $u_0(\xt)$ due to environmental force in $u$, then the cost functional can be set to $c(x,u) = \frac{1}{2}|u(x)-u_0(x)|^2$. 
In many real-world applications, the swarm with density $\rho(t,\cdot)$ is realized by a number of robotic agents. 
In such cases, it is necessary that the agents do not collide with each other throughout the control process.
To this end, one can impose additional penalty when any two robots become too close. For example, at any time $t$, such penalty function can be set to (the average should be over all agents different from the input variable $x$ but we use the following for notation simplicity):
\begin{equation}
    \label{eq:interact-potential}
    \frac{1}{N} \sum_{i=1}^{N} V(x,x_{t}^{(i)})\quad  \del{\approx \int V(x,y) \rho(t,y) \, dy},
\end{equation}
where the approximation in the parenthesis holds as $N\to \infty$, and $V$ is an interactive potential defined as $V(x,y) = c\ln |x-y|$ or $V(x,y) = c|x-y|^{-\alpha}$ for some constants $c,\alpha >0$, etc.
The purpose of $V$ is to incur larger cost when two robots become closer.
A cost functional including such interactive potential can effectively prevent robot collisions. 

\paragraph{Distance $D$ between densities}
We use the 1-Wasserstein distance, also known as the Earth Mover's Distance (EMD), as the distance $D$ to measure the difference between the terminal density $\rho_T$ and target density $\nu$.
For any pair of densities $(\varrho,\nu)$, their 1-Wasserstein distance is defined by
\begin{equation}
\label{eq:1W}
    D(\varrho,\nu) = \inf \cbr[2]{ \int |x-y| \, d \pi(x,y): \ \pi \in \Pi(\varrho,\nu)},
\end{equation}
where $\Pi(\rho,\nu)$ denotes the set of joint distributions on $\Omega \times \Omega$ whose marginal densities are exactly $\varrho$ and $\nu$.
The 1-Wasserstein distance defined in \eqref{eq:1W} quantifies the minimum total cost to transfer probability density $\varrho$ to $\nu$ (and vice versa), and its optimal solution $\pi^*$ describes how the transfer should be made.
In real-world applications, the 1-Wasserstein distance is an appropriate metric for the density control problem \eqref{eq:dc-obj}---its value indicates how much additional efforts, measured using the standard Euclidean distance, are needed to move the robots to reach the target distribution in case the matching is not exact at terminal time $T$.
Moreover, 1-Wasserstein distance allows two densities to have different supports, which could be problematic for other density distance measure such as Kullback-Leibler (KL) divergence.

Despite of the various advantages of Wasserstein distance, its computation can be challenging due to the optimization process required in \eqref{eq:1W}. However, for 1-Wasserstein distance, one can derive its dual form given by
\begin{equation}
    \label{eq:emd-dual}
    D(\rho, \nu) = \sup \{\Ebb_{\rho}[\phi] - \Ebb_{\nu}[\phi] : \ \lip(\phi) \le 1\},
\end{equation}
where $\phi$ is the dual variable and $\lip(\phi)$ denotes the Lipschitz constant of $\phi$.
As $\Omega$ is quasi-convex, we can relax the constraint $\lip(\phi)\le 1$ to $\|\nabla \phi(x)\|\le 1$ for all $x \in \Omega$ a.e., where $\nabla \phi(x)$ is the weak derivative of $\phi$ with respect to $x$, and $\|\nabla \phi(x)\|$ is the induced 2-norm of $\nabla \phi(x)$ and thus its maximal singular value. This allows us to use nonlinear reduced-order models, such as deep neural networks, to parameterize $\phi$ and apply various techniques such as spectral normalization to enforce the constraint $\|\nabla \phi(x)\|\le 1$.

\paragraph{Complete model and approximations using deep networks}
The optimization problem for solving the control problem \eqref{eq:dc-obj} is given by
\begin{equation}
    \label{eq:inf-sup-E}
    \inf_{u} \sup_{\phi} \ \Ecal(u, \phi)
\end{equation}
with constraint $\|\nabla \phi(x)\| \le 1$ to hold for $x \in \Omega$ a.e., and the objective functional $\Ecal$ is given by
\begin{equation}
    \label{eq:E-fun}
    \Ecal(u,\phi) := \Ebb\sbr[2]{\int_0^{T} c(\xt,u(\xt))\, dt} + \gamma \Ebb_{\rho_{T}}[\phi] - \gamma \Ebb_{\nu}[\phi] .
\end{equation}
As both $u$ and $\phi$ are in infinite-dimensional function spaces, we need a finite-dimensional representation of them for numerical computation. 
This can be achieved by approximating them using linear or nonlinear reduced-order models.
In particular, we will use deep neural networks to parameterize $u$ and $\phi$ as $\utheta$ and $\phieta$ respectively. Here $\theta,\eta \in \Rbb^{m}$ (we assume both are $m$-dimensional for notation simplicity here, but they can be different in practice) denote their trainable network parameters, such as their weights and biases. 
The approximation power of deep neural networks has been justified by the universal approximation theorem, which has been extended to various cases with different choices of activation functions and function approximation norm. In particular, the following approximation result ensures that $u$ can be approximated by a standard feed-forward network with rectified power unit function (RePU) $a(x)=\max(0,x)^2$ as the activation for sufficient smoothness requirement.
\begin{lemma}[Theorem 4.9 \cite{guhring2021approximation}]
Let $d \in \mathbb{N}$, $B>0$, and $\|u\|_{W^{1,\infty}(\Omega)}\leq B$. For any $\epsilon \in (0,1/2)$, there exists a feed-forward neural network $u_{\theta}$ with properly chosen network parameter $\theta$ and RePU activation such that
\[
\|u_{\theta}-u\|_{W^{1,\infty}(\Omega)} \leq \epsilon.
\]
\label{lem:guhring}
\vspace{-12pt}
\end{lemma}
Lemma \ref{lem:guhring} implies that a sufficiently large deep network is capable to represent any function in a specified $W^{1,\infty}$ ball with any prescribed accuracy $\epsilon>0$. Despite of the often pessimetic bound on the network complexity in theoretical analysis, neural networks have been verified as very powerful function approximators in numerous empirical studies in recent years. Hence neural networks have become a special type of nonlinear reduced order models to approximate general functions in infinite-dimensional spaces using parameterized functions with finite-dimensional parameters, and they have proven to be much more effective than traditional linear basis models.

In addition to using finite-dimensional representations $\utheta$ and $\phieta$, we also need to approximate the integrals and expectations in \eqref{eq:E-fun} for numerical computations.
With high dimension $d$, the Fokker-Planck equation \eqref{eq:fpe} (or the continuity equation with $\sigma=0$) cannot be computed using any classical numerical methods (e.g., finite difference and finite element methods) due to the curse of dimensionality. 
In this work, we use the particle dynamics \eqref{eq:dc-sde} if the robots are given and fixed, or simulate a large number of particles following \eqref{eq:dc-sde} otherwise, to avoid direct computation of $\rho(t,\dot)$.
More precisely, we approximate the saddle-point problem \eqref{eq:inf-sup-E} with deep neural network representations $\utheta$ and $\phieta$ and particle dynamics as follows
\begin{equation}
    \label{eq:Ehat}
    \min_{\theta} \max_{\eta} E(\theta, \eta)
\end{equation}
where the objective function given by
\begin{equation*}
    E(\theta, \eta) := \frac{1}{N}\sum_{i=1}^{N} \big(c(\xti,\utheta(\xti)) + \gamma \phieta(x^{(i)}_{T}) \big)- \gamma \Ebb_{\nu}[\phieta].
\end{equation*}
The last term in $E$ can also be replaced with empirical expectation $(\gamma/M)\cdot \sum_{j=1}^M \phieta(z^{j})$ for $M$ i.i.d.\ samples $\{z^{(j)}: j\in[M]\}$ drawn from $\nu$.
For simplicity, we also assume that $\theta,\eta \in \Theta$ and $\Theta$ is a compact set in $\Rbb^{d}$.
Then the optimal solution $\theta$ to \eqref{eq:Ehat} is the parameter of $\utheta$ approximating the optimal control $u$ in \eqref{eq:dc-obj}.
\begin{remark} 
If the target distribution $\nu$ are given as $M$ fixed discrete points $Z:=\{z^{(j)} : j \in [M]\}$, then we can also set $D$ to the Chamfer's distance between the two point clouds $X_T:=\{x^{(i)}: i \in [N]\}$ and $Z$: 
\begin{align}
\label{eq:chamfer}
\sum_{i=1}^{N} \min_{1\le j\le M} |z_T^{(i)} - z^{(j)}|^2 + \sum_{j=1}^{M} \min_{1\le i \le N} |x_T^{(i)} - z^{(j)}|^2.
\end{align}
The Chamfer's distance between $X_T$ and $Z$ given in \eqref{eq:chamfer} is a promising alternative to measure the difference between these two clouds when $N$ and $M$ are moderately small (e.g., dozens).
We will show in our experiment that the proposed method can also be applied using this Chamfer distance.
\end{remark}




\paragraph{Numerical computations}
An important feature of the saddle-point formulation \eqref{eq:Ehat} for solving the optimal control problem \eqref{eq:dc-obj} is that we can leverage parallel computation to effectively handle the large number of simulated trajectories.
More precisely, for any fixed $\theta$, we can use the Euler-Maruyama (or the Euler method if the diffusion $\sigma=0$) to generate trajectories 
\begin{equation}
    \label{eq:em}
    x_{t+h}^{(i)} = x_{t}^{(i)} + h u(x_{t}^{(i)}) + \sqrt{h} \sigma(x_{t}^{(i)}) \delta_{t}^{(i)}
\end{equation}
for discrete time points $t=0,\dots,\lfloor T/h \rfloor-1$ where $h>0$ is the time step size and $\delta_t^{(i)}$ are i.i.d.\ standard normal random variables for all $t$ and $i$.
Therefore, the objective function $E(\theta,\eta)$ given in \eqref{eq:Ehat} can be computed explicitly.
The we can use automatic differentiation to obtain the gradients of $E$ with respect to $\theta$ and $\eta$, which is crucial to solving the saddle-point problem \eqref{eq:Ehat}.
In particular, we consider the objective function
\begin{equation}
\label{eq:L}
    L(\theta):= \max_{\eta \in \Theta} E(\theta, \eta).
\end{equation}
The optimal solution $\theta$ to \eqref{eq:Ehat} is thus the minimizer of $L(\theta)$, whose full gradient with respect to $\theta$ can be immediately computed by solving the maximization problem \eqref{eq:L}, as shown in the following lemma.

\begin{lemma}[Gradient of $L$]
\label{lem:gradL}
Let $L$ be defined in \eqref{eq:L}. Then for any fixed $\theta$, the gradient $\nabla_{\theta} L(\theta)$ at $\theta$ is given by 
\begin{equation}
    \nabla_{\theta} L(\theta) = \partial_{\theta}E(\theta, \eta(\theta)),
\end{equation}
where $\eta(\theta) \in \argmax_{\eta}\{ E(\theta,\eta): \|\nabla \phieta\|\leq 1\}$.
\end{lemma}
\begin{proof}
The full gradient $\nabla_{\theta} L(\theta)$ is given by
\begin{align}
\label{eq:thm2}
    \nabla_{\theta} L(\theta) = \partial_{\theta}E(\theta, \eta(\theta)) + \partial_{\eta} E(\theta, \eta(\theta))\nabla_{\theta}\eta(\theta).
\end{align}
For any fixed $\theta$, consider the problem
\begin{equation}
\label{eq:max-E}
    \max_{\eta}\ \{ E(\theta,\eta): \| \nabla \phi(x) \| \le 1,\ \forall \, x, \ a.e.\}
\end{equation}
The Lagrangian associated with \eqref{eq:max-E} is given by
\begin{align}
    E(\theta, \eta) + \frac{1}{2}\int \lambda_{\theta}(x)(\|\nabla \phieta(x)\|^2 - 1) \,dx,
\end{align}
and the Karush-Kuhn-Tucker (KKT) system for the necessary optimality condition of $\eta(\theta)$ is given by
\begin{subequations}
\begin{align}
& \partial_{\eta_j}E(\theta, \eta(\theta)) + H_{j}(\theta, \eta(\theta)) = 0, \ \forall \, j \in [m],\label{eq:kkt1} \\
& \lambda_{\theta}(x)\geq 0,\ \|\nabla \phi_{\eta(\theta)}(x)\| \leq 1, \ \forall\, x, \ a.e.,
\label{eq:kkt2} \\
& \textstyle\int \lambda_{\theta}(x)(|\nabla \phi_{\eta(\theta)}(x)|^2 - 1)\,dx= 0. \label{eq:kkt3}
\end{align}
\end{subequations}
where the function $H_j$ is defined by
\begin{align*}
    H_{j}(\theta, \eta(\theta)) = \int \lambda_{\theta}(x) {z_{\theta}^x}^{\top}\nabla \phietatheta(x)^{\top} \partial_{\eta_{j}} \nabla \phietatheta(x) z_{\theta}^{x} \,dx, 
\end{align*}
and the vector $z_{\theta,x} \in \Rbb^{d}$ is defined by
\begin{equation*}
    z_{\theta,x} = \argmax\{|\nabla \phi_{\eta(\theta)}(x) z| : |z| \le 1\}.
\end{equation*}
Now we partition the state space $\Omega$ into the following disjoint subdomains:
\begin{align}
    \Omega_0:= &\{x: \lambda_{\theta}(x) < 0 \mbox{ or } \|\nabla \phietatheta(x)\| > 1\}, \\
    \Omega_1:= &\{x: \lambda_{\theta}(x) > 0 \mbox{ and } \|\nabla \phietatheta(x)\| < 1\}, \\
    \Omega_2:= &\{x: \lambda_{\theta}(x) =0 \}, \\
    \Omega_3:= &\{x: \lambda_{\theta}(x) > 0 \mbox{ and } \|\nabla \phietatheta(x)\| = 1\}.
\end{align}
Due to \eqref{eq:kkt2}, we know $\mu(\Omega_0) = 0$. 
Furthermore, we can see that 
\begin{equation}
\label{eq:integrand-nonpos}
\lambda_{\theta}(x) (\|\nabla \phietatheta(x)\|^2 - 1) \le 0    
\end{equation}
for all $x \in \Omega \setminus \Omega_0$. In particular, the equality in \eqref{eq:integrand-nonpos} holds for all $x$ in $\Omega_2$ and $\Omega_3$.

Now we claim that $\mu(\Omega_1) = 0$.
If not, we define for any $k,\ell \in \Nbb$ the set
\begin{equation*}
    \Omega_1^{k,\ell} := \cbr[2]{x: \lambda_{\theta}(x) > \frac{1}{k}} \cap \cbr[2]{x:\|\nabla \phietatheta(x)\|^2 < 1- \frac{1}{\ell}}.
\end{equation*}
Then it is easy to show that
\begin{equation}
    \Omega_1 = \cup_{k=1}^{\infty} \cup_{\ell=1}^{\infty} \Omega_{1}^{k,\ell},
\end{equation}
which is a countable union of sets. Hence, there must exist $k',\ell' \in \Nbb$ such that
$\mu(\Omega_{1}^{k',\ell'}) >0$. Combining this fact and \eqref{eq:integrand-nonpos}, we obtain
\begin{align*}
   &\textstyle\int \lambda_{\theta}(x) (\|\nabla \phietatheta(x)\|^2 - 1)\,dx \\
   =& \textstyle\int_{\Omega_1} \lambda_{\theta}(x) (\|\nabla \phietatheta(x)\|^2 - 1) \,dx\\
   \le& \textstyle\int_{\Omega_1^{k',\ell'}} \lambda_{\theta}(x) (\|\nabla \phietatheta(x)\|^2 - 1)\,dx \\
   \le& - \textstyle\frac{1}{k'\ell'} < 0,
\end{align*}
which contradicts to \eqref{eq:kkt3}. Therefore $\mu(\Omega_1) = 0$.

Since $\|\nabla \phietatheta(x) \|^2 = 1$ for all $x \in \Omega_3$, the following equality holds over $\Omega_3$:
\begin{align*}
    \partial_{\eta_j} \|\nabla \phietatheta(x)\|^2 = 2z_{\theta}^x{}^{\top}\nabla \phietatheta(x)^{\top} \partial_{\eta_{j}} \nabla \phietatheta(x) z_{\theta}^{x} = 0.
\end{align*}
Together with the definition of $\Omega_3$ and $\mu(\Omega_0) = \mu(\Omega_1) = 0$, we see that $H_j(\theta,\eta(\theta)) = 0$.
Combining this and \eqref{eq:kkt1}, we have $\partial_{\eta}E(\theta,\eta(\theta))=0$. Therefore by \eqref{eq:thm2} we have $\nabla_{\theta}L(\theta) = \partial_{\theta}E(\theta, \eta(\theta))$, which completes the proof.
%
\end{proof}

Lemma \ref{lem:gradL} suggests a practical implementation to compute $\nabla_{\theta} L(\theta)$: calculate the partial derivative of $E(\theta,\eta)$ with respect to $\theta$ with $\eta$ held fixed, then substitute $\eta$ by the maximizer $\eta(\theta)$ of the problem \eqref{eq:L}.
Since \eqref{eq:L} is constrained, the \textit{gradient mapping}, defined by $\Gcal(\theta):= \tau^{-1}[\theta-\Pi(\theta-\tau \nabla_{\theta} L(\theta))]$, is used as the convergence criterion of $\theta$ \cite{ghadimi2016mini-batch,li2018simple,reddi2016proximal}.
Note that the definition of gradient mapping takes the normalization of step size $\tau$ into consideration.
Moreover, the gradient mapping reduces to the gradient $\nabla_{\theta}L(\theta)$ for unconstrained case.
In the following theorem, we show that for any $\varepsilon>0$ the standard stochastic gradient descent scheme based on $\nabla_{\theta} L(\theta)$ (i.e., stochastic gradient is unbiased and has bound variance) is guaranteed to push the gradient mapping $\Gcal(\theta)$ to $0$ with $\varepsilon$ accuracy within $O(\varepsilon^{-1})$ iterations in the ergodic sense.
For simplicity, we assume both $\theta$ and $\eta$ are of the same dimension $m$ and lie in the same closed Euclidean ball in $\Rbb^{m}$. The results can be easily generalized to the case where they have different dimensions $m$ and $m'$ and are in some convex compact subsets of $\Rbb^{m}$ and $\Rbb^{m'}$ respectively.
\begin{theorem}
\label{thm:iteration}
Suppose the parameters $\theta$ and $\eta$ are bounded in a ball centered at origin with radius $R>0$, namely, we have $\Theta:=\{\theta: |\theta|\leq R\}$. For any $\varepsilon > 0$, let $\{\theta_j\}$ be a sequence of the network parameters of $u_{\theta}$ generated by the stochastic gradient descent algorithm, where $\nabla_{\theta} L(\theta)$ is approximated by the empirical expectation using samples as in \eqref{eq:Ehat}. If the sample complexity is $N = O(\varepsilon^{-1})$ in each iteration, then $min_{1\leq j \leq J}\mathbb{E}[|\Gcal(\theta_j)|^2]\leq\varepsilon$ after $J=O(\varepsilon^{-1})$ iterations.
\end{theorem}

\begin{proof}
Since $u_{\theta}$ and $\phi_{\eta}$ are parameterized using neural networks with finitely many neurons, and the parameters are bounded, we know by Lemma \ref{lem:gradL} that $\nabla L(\theta)$ is globally $M$-Lipschitz continuous on $\Theta$ for some $M>0$.
Following the procedure of the projected stochastic gradient descent (SGD) scheme, started from any initial $\theta_0$, the sequence $\{\theta_j\}$ generated by the scheme is
\begin{align}
\label{eq:th2_1}
    \theta_{j+1} = \Pi(\theta_j - \tau G_j) = \argmin_{\theta \in \Theta}(G_j^{\top}\theta + \frac{1}{2\tau}|\theta - \theta_j|^2),
\end{align}
where $G_j$ denotes the stochastic gradient of $L(\theta)$ at $\theta_j$ at the $j$th iteration using a mini-batch of samples, i.e., the average of the gradient over a mini-batch of the samples to approximate the full average \eqref{eq:Ehat}, $\Pi$ denotes the projection of $\theta$ to $\Theta$, and $\tau$ is the learning rate. Let $g_j:=\nabla_{\theta}L(\theta_j)$ denote the full (but unknown) gradient of $L$ at $\theta_j$, then we define the auxiliary sequence $\{\bar{\theta}_j\}$ generated by $g_j$ as follows:
\begin{align}
\label{eq:th2_2}
    \bar{\theta}_{j+1} = \Pi(\theta_j - \tau g_j) = \argmin_{\theta \in \Theta}(g_j^{\top}\theta + \frac{1}{2\tau}|\theta - \theta_j|^2).
\end{align}
Note that the sequence $\{\bar{\theta}\}_j$ is only for proof but not necessarily computed in practice (in fact $g_j$ can be the actual gradient of the exact integral and expectations in \eqref{eq:E-fun} and the proof below still follows).
To prove the convergence of the projected SGD iterations, first we utilize the property of $M$-Lipschitz continuity of $\nabla_{\theta}L(\theta)$ which implies that
\begin{align}
\label{eq:th2_3}
    L(\theta_{j+1}) \leq L(\theta_j) + g_j^{\top}e_j + \frac{M}{2}|e_j|^2,
\end{align}
and
\begin{align}
\label{eq:th2_4}
    -L(\bar{\theta}_{j+1}) \leq -L(\theta_j) - g_j^{\top}\bar{e}_j + \frac{M}{2}|\bar{e}_j|^2,
\end{align}
where we denote $e_j:=\theta_{j+1}-\theta_j$, $\bar{e}_j:=\bar{\theta}_{j+1}-\theta_j$ for all $i$. Due to the optimality of $\theta_{j+1}$ in \eqref{eq:th2_1}, we have
\begin{align}
\label{eq:th2_5}
    0 & \leq (G_j + \frac{\theta_{j+1}-\theta_j}{\tau})^{\top}(\bar{\theta}_{j+1}-\theta_{j+1}) \nonumber \\
    & = (G_j + \frac{e_j}{\tau})^{\top}(\bar{e}_j -e_j).
\end{align}
Combining \eqref{eq:th2_3}, \eqref{eq:th2_4} and \eqref{eq:th2_5} yields
\begin{align}
\label{eq:th2_6}
    & L(\theta_{j+1}) - L(\bar{\theta}_{j+1}) \leq (g_j - G_j)^{\top}(e_j - \bar{e}_j) \\ 
    & \qquad \qquad + \frac{e_j^{\top}(\bar{e}_j-e_j)}{\tau}+\frac{M}{2}|e_j|^2+\frac{M}{2}|\bar{e}_j|^2. \nonumber
\end{align}
Again due to the $M$-Lipschitiz continuity of $\nabla_{\theta} L(\theta)$, we have that
\begin{align}
\label{eq:thetabar-lip}
    L(\bar{\theta}_{j+1}) \leq L(\theta_j) + g_j^{\top}\bar{e}_j + \frac{M}{2}|\bar{e}_j|^2.
\end{align}
By the optimality of $\hat{\theta}_{j+1}$ in \eqref{eq:th2_2} with $g_j$, we also have
\begin{align}
\label{eq:thetabar-opt}
    0 & \leq (g_j + \frac{\bar{\theta}_{j+1}-\theta_j}{\tau})^{\top}(\theta_{j} - \bar{\theta}_{j+1}) =- (g_j + \frac{\bar{e}_j}{\tau})^{\top}\bar{e}_j.
\end{align}
Combining \eqref{eq:thetabar-lip} and \eqref{eq:thetabar-opt}, we obtain
\begin{align}
\label{eq:th2_7}
L(\bar{\theta}_{j+1}) - L(\theta_j) \leq -(\frac{1}{\tau} - \frac{M}{2})|\bar{e}_j|^2.
\end{align}
We add \eqref{eq:th2_6} and \eqref{eq:th2_7} which yields
\begin{align}
\label{eq:th2_8}
    & L(\theta_{j+1}) - L(\theta_j) \leq (g_j - G_j)^{\top}(e_j - \bar{e}_j)  \\ 
    & \qquad \qquad \quad + \frac{e_j^{\top}(\bar{e}_j-e_j)}{\tau}+\frac{M}{2}|e_j|^2-(\frac{1}{\tau} - M)|\bar{e}_j|^2. \nonumber
\end{align}
Now due to the Cauchy-Schwarz inequality and the definitions of $\theta_{j+1}$ and $\bar{\theta}_{j+1}$ in \eqref{eq:th2_1} and \eqref{eq:th2_2}, we show that
\begin{align}
    & (g_j - G_j)^{\top}(e_j - \bar{e}_j) \nonumber \\
    &= (g_j - G_j)^{\top}(\theta_{j+1} - \bar{\theta}_{j+1}) \nonumber  \\
    & = (g_j - G_j)^{\top}(\Pi(\theta_j-\tau G_j)- \Pi(\theta_j - \tau g_j)) \label{eq:th2_9}\\
    &\leq |g_j - G_j||\Pi(\theta_j-\tau G_j)- \Pi(\theta_j - \tau g_j)| \nonumber \\
    & \leq \tau|g_j-G_j|^2, \nonumber 
\end{align}
where the last inequality is due to the fact that the projection $\Pi$ onto the convex set $\Theta$ is non-expansive. Moreover we notice that
\begin{align}
\label{eq:th2_10}
    \frac{e_j^{\top}(\bar{e}_j-e_j)}{\tau} = \frac{1}{2\tau}(|\bar{e}_j|^2 - |e_j|^2 - |e_j - \bar{e}_j|^2).
\end{align}
Substituting \eqref{eq:th2_9} and \eqref{eq:th2_10} into \eqref{eq:th2_8} we have
\begin{align}
\label{eq:th2_11}
    L(\theta_{j+1}) - L(\theta_j) \leq \tau|g_j-G_j|^2 - (\frac{1}{2\tau}-M)|\bar{e}_j|^2
    \\ \quad -(\frac{1}{2\tau}-\frac{M}{2})|e_j|^2-\frac{1}{2\tau}|e_j-\bar{e}_j|^2. \nonumber
\end{align}
We denote $\sigma^2$ the variance of the stochastic gradient $G_j$, i.e., $\mathbb{E}[|G_j - g_j|^2] \leq \varepsilon$. Note that by the sampling complexity of Monte-Carlo integration there is $\sigma^2 =C N^{-1}$ for some $C>0$ dependent on the maximum value of $L(\theta)$ (which is bounded due to the continuity of $L$ and boundedness of $\Theta$) but not on $N$. 
Taking expectation on both sides of \eqref{eq:th2_11} and reordering lead to
\begin{align}
\label{eq:th2_13}
    &(\frac{1}{2\tau}-M)\mathbb{E}[|\bar{e}_j|^2]\leq \mathbb{E}[L(\theta_j)]-\mathbb{E}[L(\theta_{j+1})] \\ &\qquad  \qquad \qquad +\tau \sigma^2 - (\frac{1}{2\tau}-\frac{M}{2})\mathbb{E}[|e_j|^2], \nonumber
\end{align}
Recall the definition of the gradient mapping, we have
\begin{align}
    \Gcal(\theta_j) &= \tau^{-1}[\theta_{j} - \Pi(\theta_{j} - \tau g_j)] \label{eq:gmap}\\
    & = \tau^{-1}(\theta_{j} - \bar{\theta}_{j+1}) = -\tau^{-1}\bar{e}_j. \nonumber
\end{align}
Combining \eqref{eq:th2_13} and \eqref{eq:gmap}, taking sum over $j=1,2,...,J$, dividing both sides by $(\frac{1}{2\tau}-M)\tau J$, setting $\tau = \frac{1}{4M}$, and noticing $\sigma^2 = C/N$, we obtain
\begin{align}
    \label{eq:th2_14}
    \min_{1\le j \le J} \mathbb{E}&[|\Gcal(\theta_j)|^2] \le \frac{1}{J}\sum_{j=1}^J\mathbb{E}[\Gcal(\theta_j)|^2] \\
    &\leq\frac{16M(\mathbb{E}[L(\theta_1)]-\mathbb{E}[L(\theta_{J+1})])}{J} + \frac{4C}{N}.  \nonumber
\end{align}
Notice that $\mathbb{E}[L(\theta_{J+1})] \ge L^*:= \min_{\theta \in \Theta} L(\theta)$. By choosing $N=J=[16M(\Ebb[L(\theta_1)-L^*)+4 C]\varepsilon^{-1}=O(\varepsilon^{-1})$, we obtain that
$\min_{1\le j \le J} \mathbb{E}[|\Gcal(\theta_j)|^2] \le \varepsilon$, as claimed.
\end{proof}
In practice, we found that using a small number of iterations to solve the inner maximization problem with any fixed $\theta$ appears to perform the best.
The pseudo code to solve the saddle-point problem \eqref{eq:Ehat} is given in Algorithm \ref{alg:large_n}. The symbol $\hat{\nabla}$ and $\hat{\partial}$ stand for stochastic gradient and partial derivative using mini-batch samples of $X_{t}$.
The selection of parameters will be given in Section \ref{sec:result}.
\begin{algorithm}[ht]
\caption{Optimal Density Control (ODC)}
\label{alg:large_n}
\begin{algorithmic}[1]
\STATE {\bfseries Input:} Samples $X_{0} = \{x_{0}^{(i)}:i\in [N]\}$ following initial distribution $\rho_{0}$, samples following target distribution $\nu$, time horizon $T > 0$, step size $h$, maximum outer and inner iteration numbers $K_1$ and $K_2$. 
\STATE {\bfseries Initialize:} Neural networks $\utheta, \phieta$.
\FOR{$k_{out}=1,\dots,K_1$}
\STATE Generate trajectories $X_{t} = \{x_{t}^{(i)}:i\in [N]\}$ using \eqref{eq:em} with $0 \le t \le T$.
\STATE Compute running cost $E(\theta,\eta)$ in \eqref{eq:Ehat}.
\FOR{$k_{inn} = 1,\dots, K_2$}
\STATE Update $\eta \leftarrow \eta + \tau \hat{\nabla}\{\mathbb{E}_{\rho_{X_T}}[\phi_{\eta}]-\mathbb{E}_{\nu}[\phi_{\eta}]\}$.
\ENDFOR
\STATE Update $\theta \leftarrow \theta - \tau \hat{\partial_{\theta}} E(\theta,\eta)$.
\ENDFOR
\STATE {\bfseries Output:} $\utheta$.
\end{algorithmic}
\end{algorithm}
\vspace{-12pt}


\section{EXPERIMENTS}
\label{sec:result}
In this section we validate Algorithm \ref{alg:large_n} (ODC) on both synthetic and realistic data sets. For specific examples, we also test Algorithm \ref{alg:large_n} using Chamfer distance \ref{eq:chamfer}, refered to as ODC-Chamfer. Note that it is easier to implement ODC-Chamfer since the Chamfer distance has a closed-form expression and hence the inner iteration \ref{alg:large_n} of ODC can be eliminated.

\subsection{Synthetic Data}
We first evaluate our model on several synthetic data sets, termed by \textbf{Synthetic-1}, \textbf{Synthetic-2} and \textbf{Synthetic-3}. The state spaces of Synthetic-1 and Synthetic-2 are 3 and 100 dimensional, while Synthetic-3 is is 2 dimensional respectively. 
We set the control $\utheta = \nabla \psi_{\theta}$, where $\psi_{\theta}$ is a fully connected neural network with 3 hidden layers and 36 nodes per layer. We also set the dual function $\phieta$ as a fully-connected neural network with 6 hidden layers and 256 nodes per layer. We use tanh as the activation function for all layers. We implement and test Algorithm \ref{alg:large_n} using PyTorch and train the networks using the builtin optimizer Adam \cite{kingma2014adam} with learning rate $10^{-4}$. This optimizer appears to have improved efficiency compared to SGD empirically. Furthermore, we use spectral normalization to enforce $\lVert \nabla \phieta \rVert \leq 1$ \cite{spectral}. We initialize all network parameters with Xavier initialization \cite{xvaier} and train $\utheta$ according to Algorithm \ref{alg:large_n} ODC (or ODC-Chamfer if the number of agents is small and the target locations are fixed). When ODC is used, we set the size of simulated agents at each time point as $N=2,000$, and use $1,500$ points as the training set and the other $500$ data points for testing. The step size $h$ and time horizon $T$ are set to $0.1$ and $1$ respectively, we choose outer iteration number as $10^4$ and inner iteration number for approximating Wasserstein distance as $6$, for all synthetic experiments. 

\vspace{-9pt}
\noindent\textbf{Synthetic-1:}
We first apply Algorithm \ref{alg:large_n} ODC with noise term in a state space of dimension $d=3$ for two instances (two different pairs of initial-target densities), denoted by Syn-1-a and Syn-1-b, respectively. The initial and target distributions are set to mixtures of Gaussians. 
%
%
%
We set the perturbation noise $\sigma=0.01$. 
The controlled density and target density are respectively shown as the green and blue point clouds in Figure \ref{fig:syn1}, which shows that the controlled density in green correctly moved from the initial at $t=0$ to the blue target density at $t=1$ in the presence of perturbations in both Syn-1-a and Syn-1-b.
%

\begin{figure}[ht!]
    \centering
     \subfloat[][Syn-1-a:t=0]{\includegraphics[width=.25\linewidth]
     {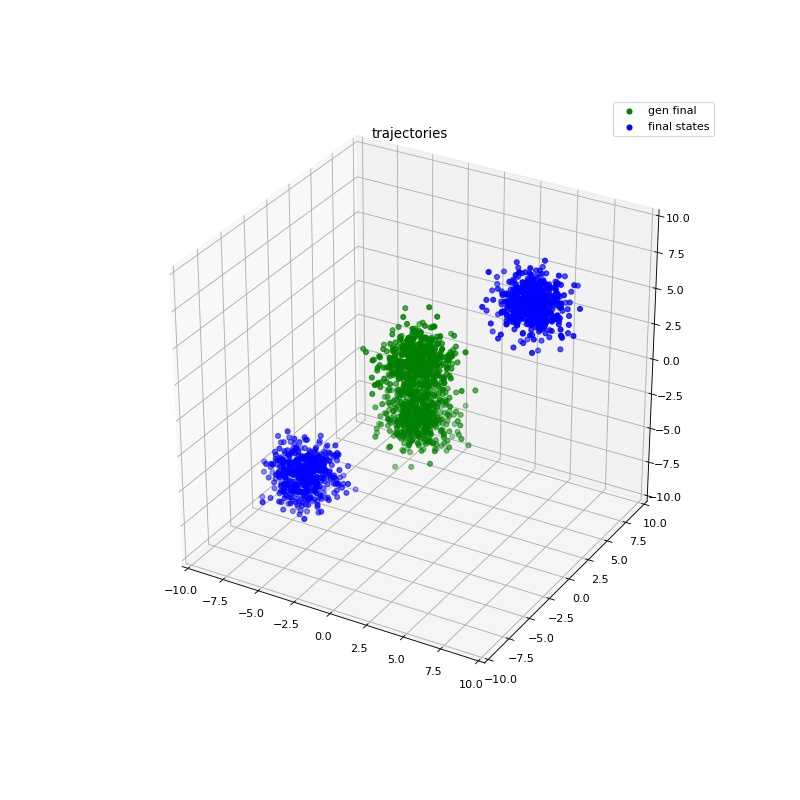}}
     \subfloat[][Syn-1-a:t=1]{\includegraphics[width=.25\linewidth]{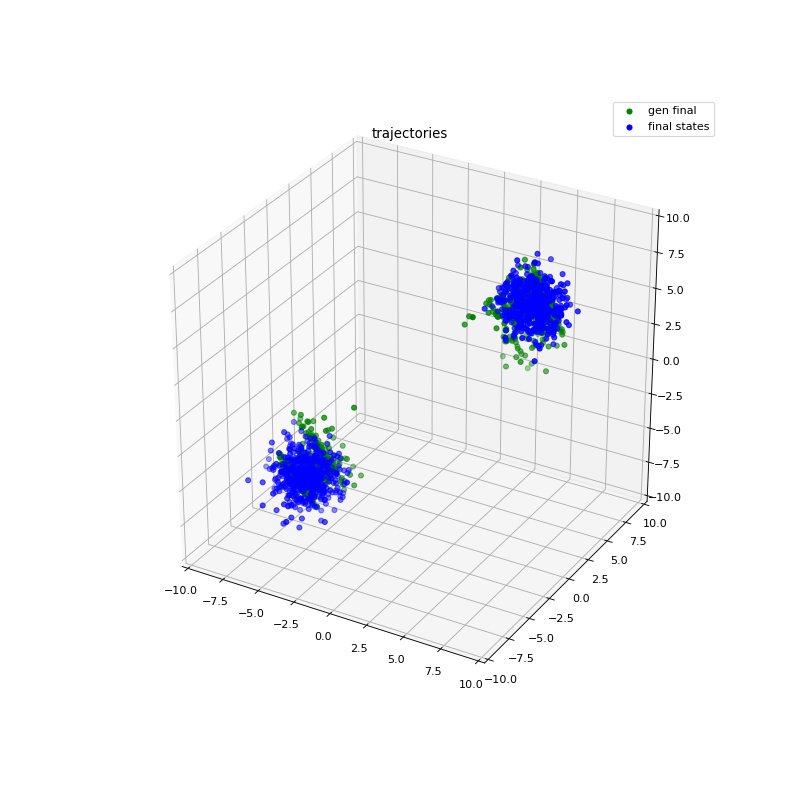}}
     \subfloat[][Syn-1-b:t=0]{\includegraphics[width=.25\linewidth]{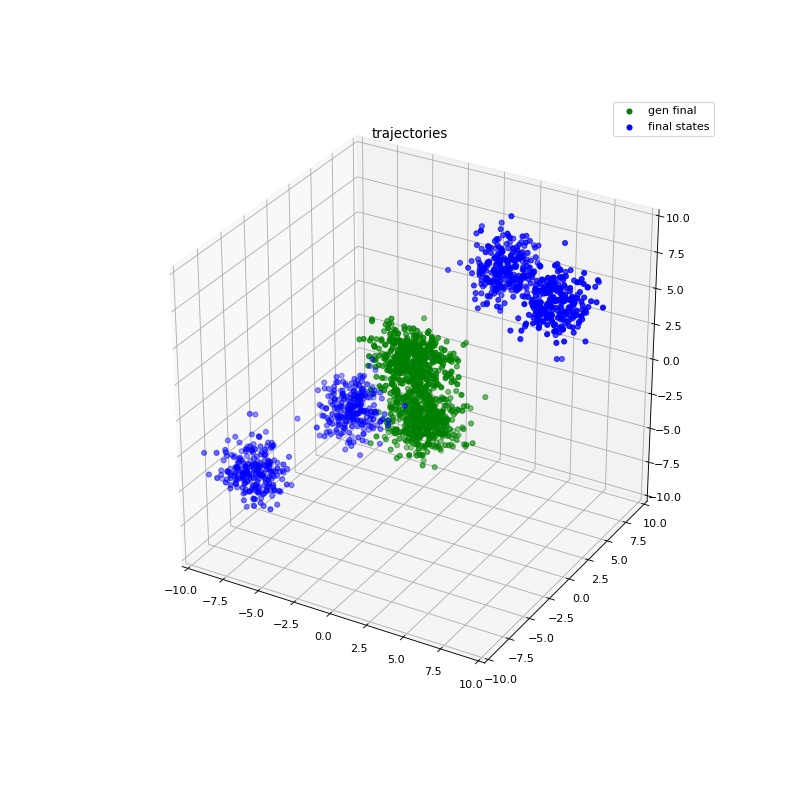}}
      \subfloat[][Syn-1-b:t=1]{\includegraphics[width=.25\linewidth]{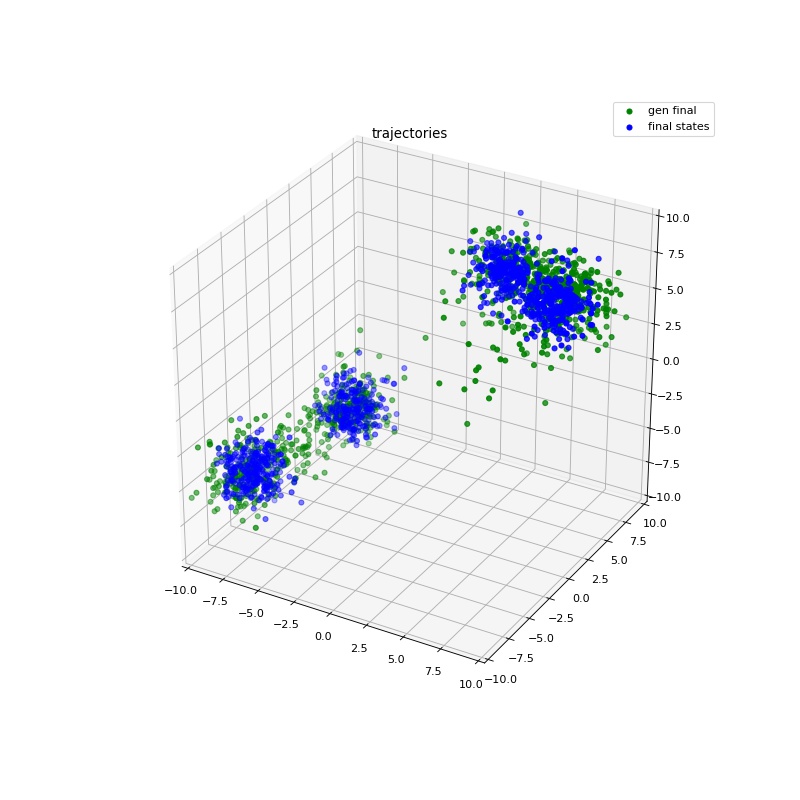}}
\caption{Results of ODC on Synthetic-1 datasets for two instances. The green point clouds indicate the controlled density, and the blue cloud points indicate the target density. (a) and (b) show the densities at $t=0$ and $t=1$ respectively for the first instance. (c) and (d) show for the second instance.}
\label{fig:syn1}
\end{figure}

\vspace{-9pt}
\noindent\textbf{Synthetic-2:}
We test the control using a large number of agents in a state space of dimension $d=100$. We simulate 2 pairs of initial-target densities and name them as Syn-2-a and -b. For Syn-2-a and Syn-2-b, we plot the controlled density in green point clouds and target in blue as above, and show the projection of them onto the first three coordinates for $t=0$ and $t=1$ in Figure \ref{fig:syn2} (a) and (b) respectively.

\begin{figure}[ht!]
    \centering
     \subfloat[][Syn-2-a:t=0]{\includegraphics[width=.25\linewidth]{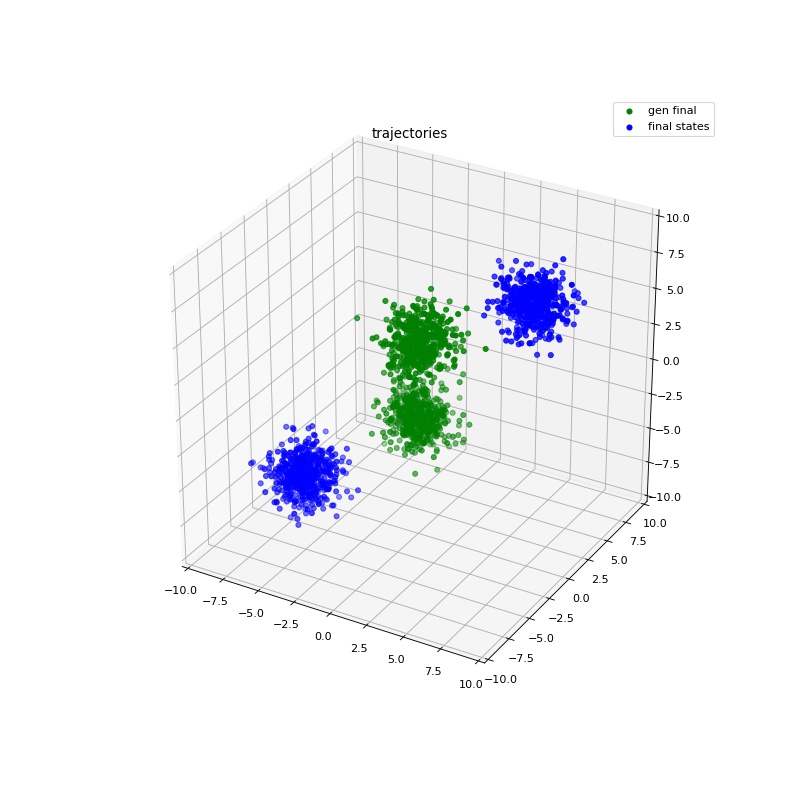}}
     \subfloat[][Syn-2-a:t=1]{\includegraphics[width=.25\linewidth]{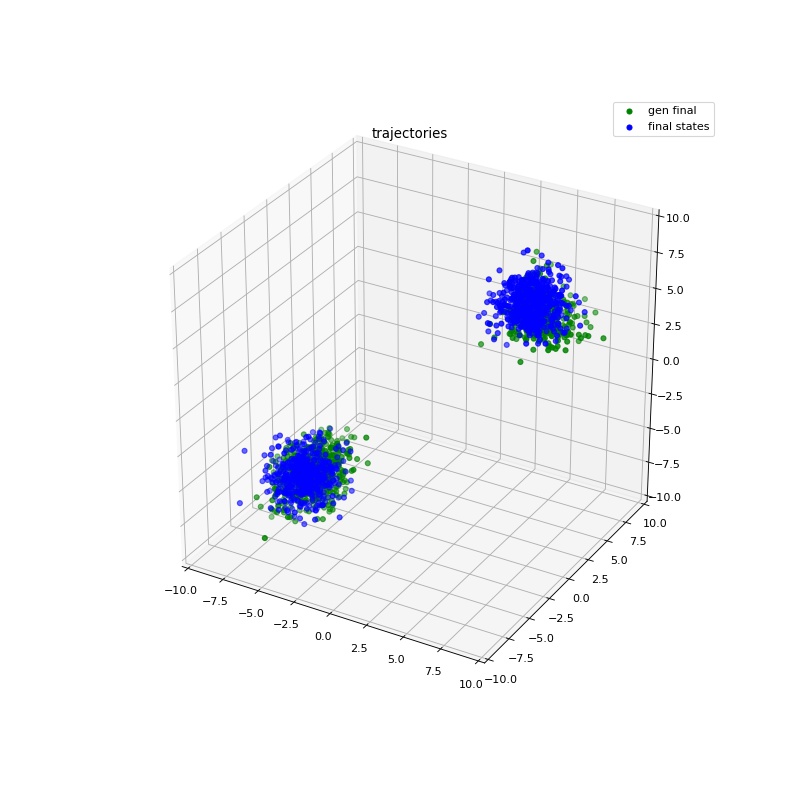}}
     \subfloat[][Syn-2-b:t=0]{\includegraphics[width=.25\linewidth]{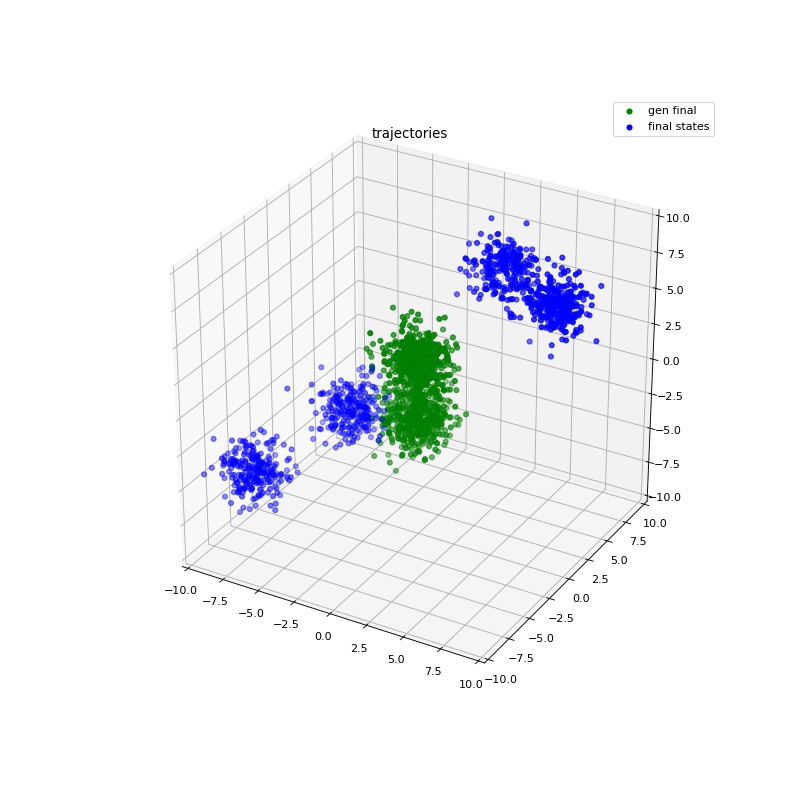}}
      \subfloat[][Syn-2-b:t=1]{\includegraphics[width=.25\linewidth]{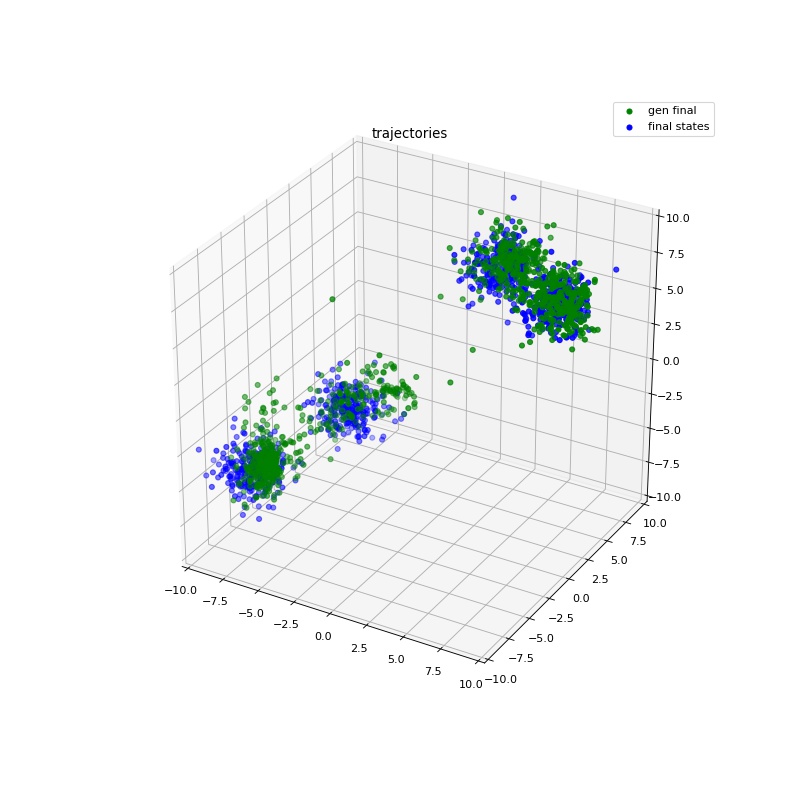}}\\
    \vspace{-3pt}
\caption{Results of ODC on Synthetic-2 datasets for four instances of dimension 100 as projections onto the first three dimensions (top row) and first two dimensions (bottom row). Note that we used ResNet and Normalizing Flow to parameterize the control $u$ in Syn-2-a and Syn-2-b respectively in the bottom row.}
\label{fig:syn2}
\end{figure}

\begin{figure}[ht!]
    \centering
     \subfloat[][Syn-3:t=0]{\includegraphics[width=.25\linewidth]
     {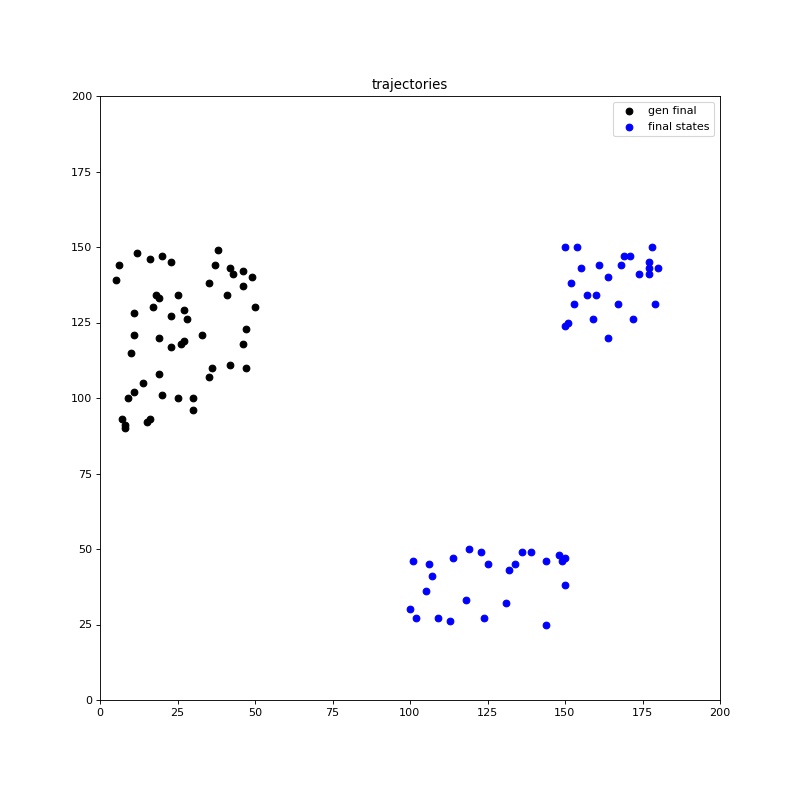}}
     \subfloat[][Syn-3:t=10]{\includegraphics[width=.25\linewidth]{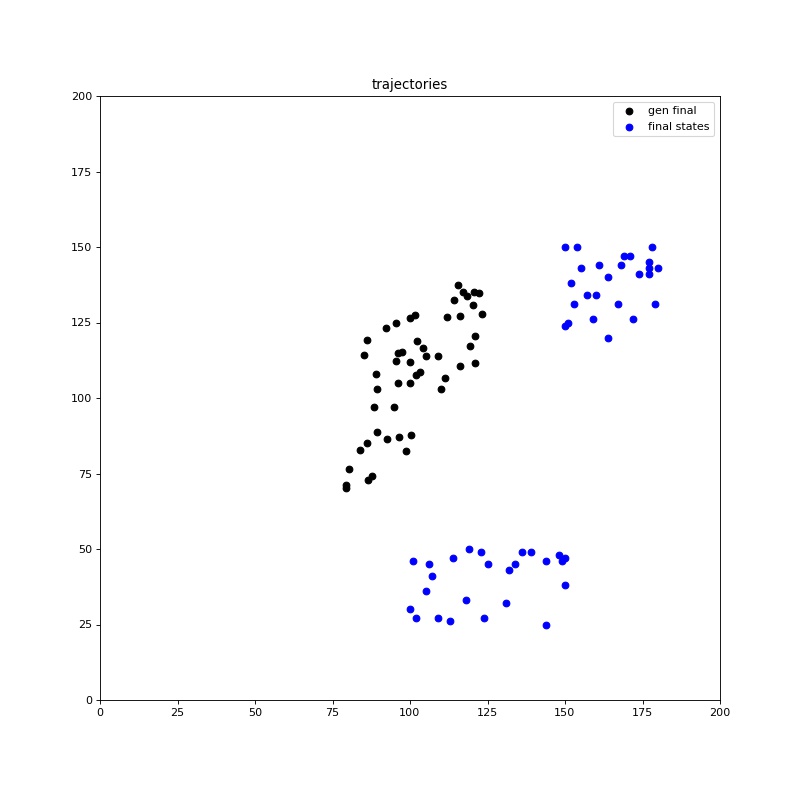}}
     \subfloat[][Syn-3:t=20]{\includegraphics[width=.25\linewidth]{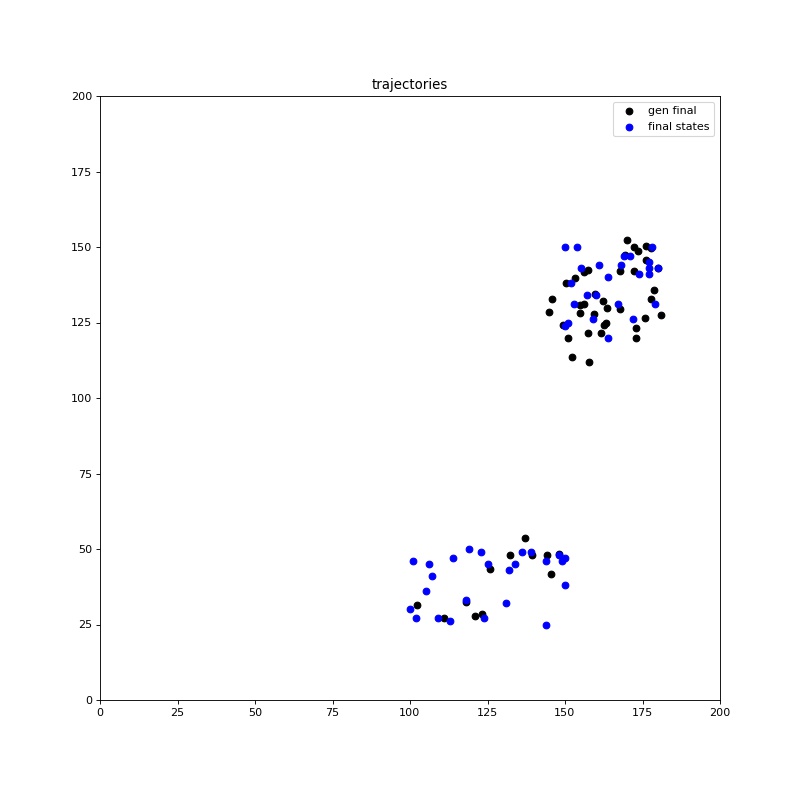}}
      \subfloat[][Closest Dist]{\includegraphics[width=.25\linewidth]{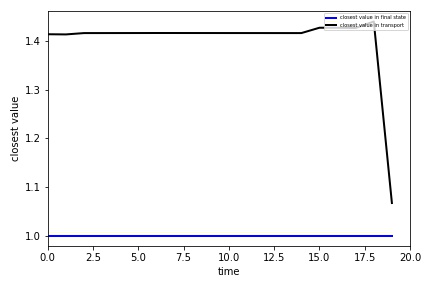}}
\caption{Results of ODC on Synthetic-3 to avoid collison on pairwise agent distance. The black point clouds indicate the controlled density, the blue cloud points indicate the target density. The agents move towards target distribution from (a) to (c) and in (d) shows the closest pairwise distance among agents at different time steps.}
\label{fig:syn3}
\vspace{-12pt}
\end{figure}

\vspace{-9pt}
\noindent\textbf{Synthetic-3:} We also test Algorithm \ref{alg:large_n} (ODC) with an interaction potential with $V(x,y)=|x-y|^{-2}$ in \eqref{eq:interact-potential} to avoid agent collision on a 2D example. The agents are indicated by the black points and the target distribution is indicated by the blue points in Figure \ref{fig:syn3}. The left three panels in Figure \ref{fig:syn3} show them at initial time $t=0$, middle time $10$, and terminal time $20$. The rightmost figure show the minimum pairwise distance of the controlled agents stays around 1.4 until near the terminal time. This minimum distance is consistently above the minimum pairwise distance 1 of the target sample points, showing that the learned control tends to avoid collisions between the agents during the movement.

\begin{figure}[ht!]
    \centering
      \subfloat[][Real: t=0]{\includegraphics[width=.27\linewidth]{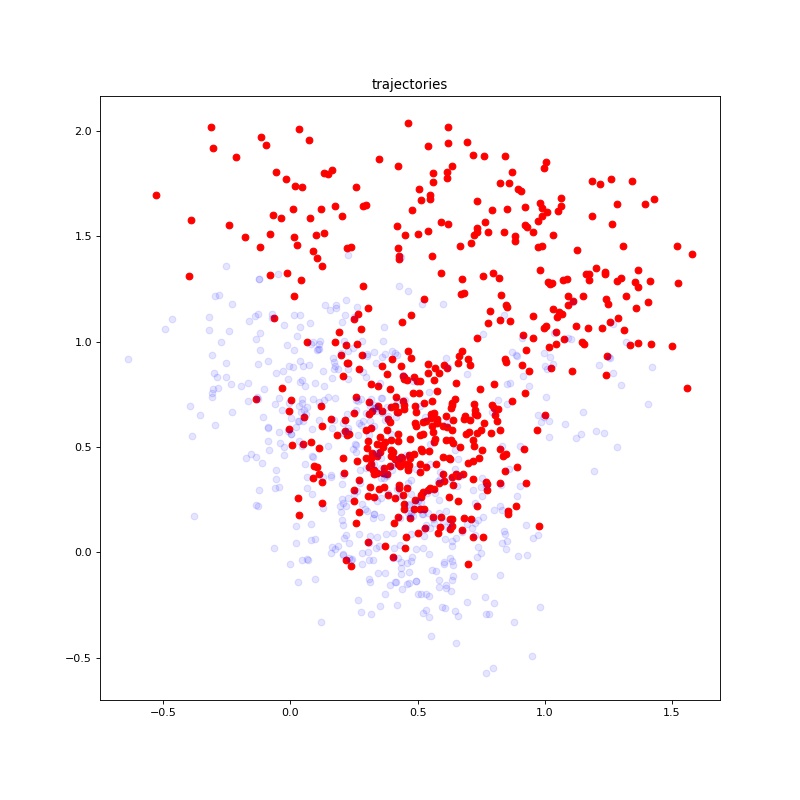}}
     \subfloat[][Real: t=0.5]{\includegraphics[width=.27\linewidth]{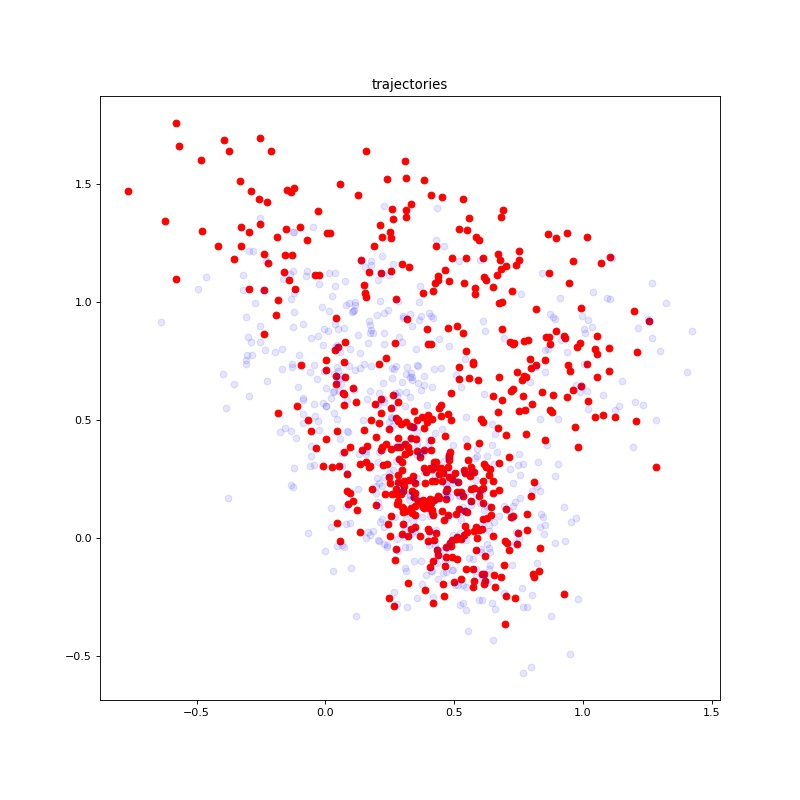}}
     \subfloat[][Real: t=1]{\includegraphics[width=.27\linewidth]{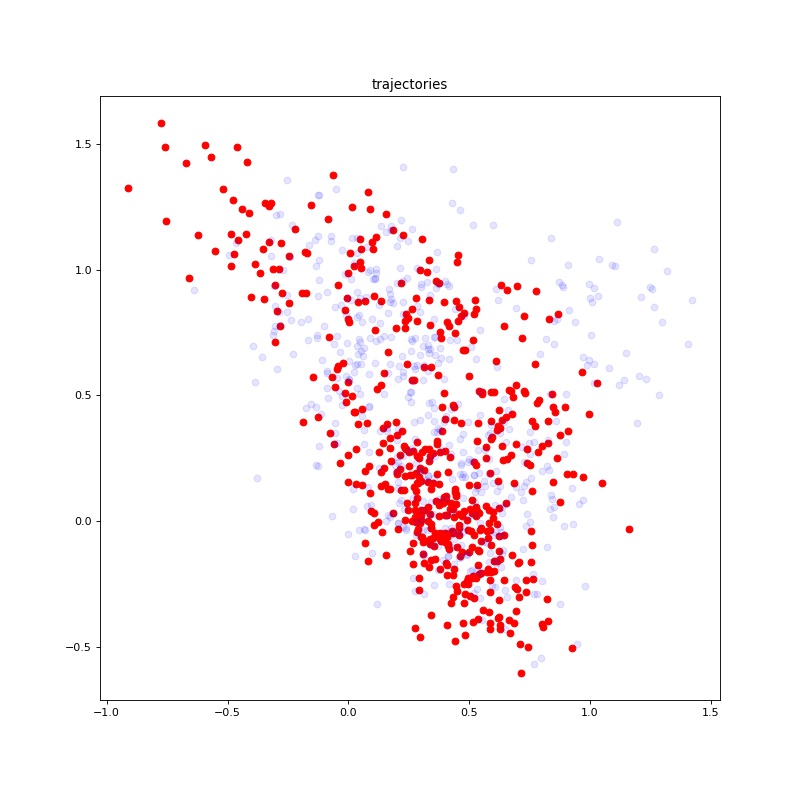}}
\caption{Results on real data. The controlled density (red) and the target density (blue) are shown at normalized times (a) $t=0$, (b) $t=0.5$ and (c) $t=1$.}
\label{fig:real1}
\end{figure}
\vspace{-12pt}

\subsection{Real Data}
In this experiment we aim to control a group of Autonomous Underwater Vehicles (AUVs) to reach a target distribution. This experiment is motivated by the 16-month AUV ocean deployment (Processes driving Exchange At Cape Hatteras, PEACH) near Cape Hatteras, North Carolina, a highly dynamic region characterized by confluent western boundary currents and convergence in the adjacent shelf and slope waters. Due to the AUVs limited forward speed, the motion of AUV is highly influenced by the ocean flow field $v$. We define the cost functional as $r(x,u) = \frac{1}{2} \Ebb[\int_{0}^{T}|u(x_t) - v(x_t)|^2dt]$.
The input flow map in this simulation is given by a 1-km horizontal resolution version of the Navy Coastal Ocean Model \cite{martin2000} made available by J.\ Book and J.\ Osborne (Naval Research Laboratory, Stennis Space Center). In the AUV experiment, multiple vehicles were deployed repeatedly in the same domain, to sample the variability in the position of the Hatteras Front. Since the AUVs were deployed at similar locations throughout the $16$ months experiment, we collect the starting and goal positions from all experiments as the initial and target distribution of the system.   
We aim to move the AUVs to the surround of target location on distribution level. Since the data size is small ($10$ samples only), we augment the data by generating Gaussian samples around the given locations, thus the final distributions are mixture of Gaussians. The we apply Algorithm \ref{alg:large_n} (ODC) to the data and show the controlled density (red) and target density (blue) at normalized time $t=0,\ 0.5, 1.0$ in Figure \ref{fig:real1}. The network $\utheta$ is parameterized as a ResNet with three layers and each layer has 36 nodes.
As we can observe, the initial density gradually moves to the target density under the guidance of the learned control $\utheta$.
%

\section{CONCLUSION}
\label{sec:conclusion}
In this paper, we develop a novel computational framework for density control. We formulate the control problem with running cost and penalize inaccurate matching to target density using Wasserstein metric. The control problem is reformulated as a saddle-point optimization problem of the control and the dual function of the Wasserstein distance. Both of the control and the dual function are parameterized as neural networks which can handle high-dimensional problems without any spatial discretization or basis representations. We validated the promising performance of this method on several synthetic and real data sets. 


%


%
%
%
%
%

\bibliographystyle{abbrv}
\bibliography{odc_ref}

\end{document}